\documentclass[twoside]{article}
\usepackage[usenames,dvipsnames]{color}
\usepackage{graphicx}
\usepackage{camnum}
\usepackage{harvard}
\usepackage{amsmath,amssymb}
\usepackage{caption} 
\usepackage{cite}
\usepackage{geometry}[margins=1in]

\usepackage{tikz,pgfplots}
        \pgfplotsset{compat = 1.3}
        \pgfplotsset{minor grid style={dotted}} \pgfplotsset{major grid
        style={dashed}}

        \pgfplotsset{every x tick label/.append style={font=\footnotesize,
        yshift=0.25ex}}

        \pgfplotsset{every y tick label/.append
        style={font=\footnotesize, xshift=0.25ex}}

\definecolor{colorclassyorange}{rgb}{0.95000,0.32500,0.09800}
\definecolor{colorclassyblue}{rgb}{0.00000,0.44706,0.74118}
\definecolor{colorpurple}{rgb}{0.49400,0.18400,0.55600}
\definecolor{colorimag}{rgb}{0.00000,0.49804,0.00000}

\newtheorem{remark}{Remark}
\counterwithin{table}{section}
\numberwithin{equation}{section}

\begin{document}

\title{Exponentially fitted methods with a local energy conservation~law}

\author{Dajana Conte, Gianluca Frasca-Caccia\footnote{Corresponding author: gfrascacaccia@unisa.it}}
\date{\normalsize Department of Mathematics, University of Salerno,\\ Via Giovanni Paolo II n. 132, 84084 Fisciano (SA), Italy}
\maketitle

\abstract A new exponentially fitted version of the Discrete Variational Derivative method for the efficient solution of  oscillatory complex  Hamiltonian Partial Differential Equations is proposed. When applied to the nonlinear Schr\"odinger equation, the new scheme has discrete conservation laws of charge and energy. The new method is compared with other conservative schemes from the literature on a benchmark problem whose solution is an oscillatory breather wave.

\section{Introduction}
Let us consider a Hamiltonian Partial Differential Equation (PDE) for a real or complex variable $z=z(x,t)$ in the form
\begin{equation}\label{genHam}
\frac{\partial z}{\partial t}=\mathcal{J}\frac{\delta \mathcal H}{\delta z^*},
\end{equation}
where $z^*$ is the complex conjugate of $z$, if $z\in\mathbb{C}$, or $z^*=z$, if $z\in\mathbb{R}$, $\mathcal{J}$ is a skew-adjoint operator independent of $z$, and $\mathcal{H}$ is a Hamiltonian functional,
\begin{equation}\label{Ham}
\mathcal{H}=\int H(z,z_x) \mathrm{d}x,
\end{equation}
where $H$ is a real local energy function. The operator on the right hand side of \eqref{genHam} is the variational derivative of $\mathcal{H}$ defined by the Euler-Lagrange expression. When applied to the functional \eqref{Ham} it reduces to
\begin{equation}\label{varder}
\frac{\delta \mathcal{H}}{\delta z^*}=\frac{\partial H}{\partial z^*}-\frac{\mathrm d}{\mathrm d x}\frac{\partial H}{\partial z_x^*}.
\end{equation}
When $z$ is a complex variable, equation \eqref{genHam} is typically complemented by its complex conjugate equation,
$$\frac{\partial z^*}{\partial t}=\mathcal{J}^*\frac{\delta \mathcal H}{\delta z}.$$
However, for real-valued Hamiltonians these two equations are equivalent and the latter can be dropped without loss of information \cite{Bodu2}.

The study of Hamiltonian PDEs has attracted the attention of numerical analysts for decades, and a wide range of numerical methods with the property of conserving invariants of the continuous problem has been developed.

Methods that conserve global invariants are usually preferable on one hand because this is a property of the exact solutions that is desirable to preserve. On the other hand, for their superior accuracy over long times. In fact, while for non conservative methods the solution error grows quadratically in time, this drift is only linear for conservative methods \cite{Frut,Duran1,Duran2}.

An invariant that all Hamiltonian PDEs have, is the Hamiltonian functional $\mathcal{H}$ itself.
Numerical methods that conserve the Hamiltonian can be obtained by applying a space discretization that defines a system of ODEs whose Hamiltonian function approximates functional $\mathcal{H}$. An energy conserving method for ODEs is then applied for the time discretization. Popular techniques to derive energy-conserving time integrators include line integral methods \cite{Limbook,Brugnano2019,BruMon} and discrete gradient methods \cite{Celle,Dahlby,McLachlan,Rob,Gonz}.  One of the most studied energy-conserving methods is the Average Vector Field (AVF) method and it can be derived from both these two approaches. The AVF method was first introduced in \cite{Quisp}, and despite its simplicity has important properties of linear covariance and preservation of linear symmetry \cite{Celle}. 

A different technique to derive energy conserving methods for Hamiltonian PDEs is the Discrete Variational Derivative method. In this approach a discrete counterpart of the variational derivative is applied to a space approximation of the Hamiltonian functional, yielding a scheme that conserves the semidiscrete energy \cite{Furi,furibook,Matsuo,Mats}. 

The conservation of the Hamiltonian, such as of any other global invariant of a PDE, is obtained from the integration in space of a local conservation law provided that the boundary conditions assigned to the problem satisfy suitable conservative assumptions (e.g., periodicity). Conservation laws are total divergences,
\begin{equation}\label{CLAw}
D_x F+D_t G,
\end{equation}
that vanish when evaluated on solutions of the PDE. Functions $F$ and $G$ are called flux and density, respectively, and may depend on the independent variables, the dependent variable and its partial derivatives.

Since conservation laws are local properties, a numerical method must satisfy stronger constraints to preserve them. Moreover, they hold true on any smallest part of the domain and are satisfied by the solutions of the differential equation regardless of the boundary conditions. 

McLachlan and Quispel have proved that discrete gradient methods preserve the energy conservation law of the space discretization, if any \cite{McLachlan}.
More recently, a strategy to derive in a systematic way bespoke finite difference schemes that preserve multiple conservation laws has been proposed in \cite{Frasca-Caccia2021,IMA} and used in \cite{Frasca-Caccia2021,IMA,mKdV,AMC} to obtain methods with local conservation laws of energy and of mass or charge. 

Although all these integrators typically perform better than standard methods, they require very small stepsizes in order to correctly reproduce the oscillations of a highly oscillatory solution.

When the oscillatory behaviour of the solution is known a priori, exponentially fitted (EF) methods can be used to solve the problem in an accurate and efficient way. EF methods are obtained by requiring exactness for functions that belong to a specific fitting space, whose choice depends on the expected behaviour of the solution \cite{Pater,Ixaru}. For example, a method that is exact for all functions in the space generated by
$$\{\cos(\omega t),\sin(\omega t)\},$$
is expected to approximate periodic solutions that oscillate with frequency $\omega$ better than a standard method, particularly for large values of $\omega$ \cite{Miya}.  The chance of making a convenient choice of the fitting space is based on the prior knowledge of the frequency of oscillation, $\omega$.  However, when unknown, the frequency can be estimated by using one of the many approaches suggested in literature \cite{DEP,VIM,VanDaele}.

Exponential fitting techniques have been successfully used to solve problems of very different nature, such as fractional differential equations \cite{BCDP}, quadrature \cite{CPquad,EW,ConteJCAM}, interpolation \cite{MVB}, time and space integrators for ODEs \cite{Conte1,Simos,Conte2} and PDEs \cite{CFC,DP,Card}, integral equations \cite{CIP,CIPS}, boundary value problems \cite{VanD}.

This paper focuses on schemes that have a local conservation law of the energy, and that are an EF version of the AVF method or the DVD method. An EF version of the AVF method has been introduced by Miyatake in \cite{Miya}. We show that this method has the same local energy conservation law of the classic AVF method.

For many important Hamiltonian PDEs (e.g., Korteweg de Vries equation) the AVF method and the DVD method lead to the same schemes \cite{Dahlby}. We show that when they are applied to the nonlinear Schr\"odinger (NLS) equation, they yield two different schemes. Therefore, we propose a new EF version of the DVD method in \cite{Matsuo} for complex Hamiltonian evolution equations in the form
\begin{equation}\label{HamPDE}
\frac{\partial z}{\partial t}=-\mathrm{i}\frac{\delta \mathcal H}{\delta z^*},\qquad (x,t)\in(a,b)\times(0,T).
\end{equation}
The new EF DVD method and the standard DVD method applied to \eqref{HamPDE} conserve the same global energy.

We apply the AVF method, the DVD method, and their EF versions to the NLS equation and demonstrate that although they are all different schemes, they all conserve the same local conservation law of the Hamiltonian. 

Moreover, the DVD method and the EF DVD method have also a local conservation law of charge. Although these conservation laws are different, they imply conservation of the same discrete global charge  when the boundary conditions are conservative.

With these premises, this paper is organised as follows. In Section~\ref{DVDsec} we first describe the DVD method in \cite{Matsuo} for complex Hamiltonian PDEs \eqref{HamPDE}. Then we introduce the new EF version of this method, showing that both schemes conserve the same semidiscrete global energy. In Section~\ref{AVFsec} we describe the AVF method in \cite{Quisp} for equation \eqref{genHam} and its EF version introduced in \cite{Miya}, and we show that these methods have the same local conservation law of the energy. In Section~\ref{NLSsec} we apply all these methods to the NLS equation, and we give explicit expressions of their conservation laws and of their invariants. In Section~\ref{testsec} a highly oscillatory breather wave solution of the NLS equation is taken as a benchmark problem to test the properties of convergence and conservation of the considered schemes and to compare their accuracy. Finally, we draw some conclusive remarks in Section~\ref{concl}.
\section{Discrete Variational Derivative method}\label{DVDsec}
We begin this section by defining the discrete operators that are used throughout this paper. We first introduce a uniform grid with nodes
\begin{align*}
x_m=&\,a+(m-1)\Delta x,\qquad m=1,\ldots,M,\qquad \Delta x=\frac{b-a}{M-1},\\
t_n=&\,n\Delta t,\qquad n=0,\ldots,N,\qquad \Delta t=\frac{T}{N},
\end{align*}
and the vectors $\mathbf{Z}\in\mathbb{R}^M$ and $\mathbf{z}_n$ of the approximations
$$Z_m(t)\approx z(x_m,t),\qquad t\in(0,T),\qquad (\mathbf{z}_n)_m=z_{m,n}\approx z(x_m,t_n),$$
respectively. Moreover, we define the difference operators
$$\delta_m^+Z_m=\frac{Z_{m+1}-Z_m}{\Delta x},\qquad \delta_m^-Z_m=\frac{Z_{m}-Z_{m-1}}{\Delta x},	\qquad \delta_m^{(2)}Z_m=\frac{Z_{m+1}-2Z_m+Z_{m-1}}{\Delta x^2},$$
acting similarly on the first index when applied to $z_{m,n}$, and the time difference operator and average operators,
$$\delta_n^+z_{m,n}=\frac{z_{m,n+1}-z_{m,n}}{\Delta t},\qquad \mu_nz_{m,n}=\frac{z_{m,n+1}+z_{m,n}}{2},\qquad \mu_mz_{m,n}=\frac{z_{m+1,n}+z_{m,n}}{2}.$$
As in \cite{Matsuo}, we introduce the DVD method assuming that the local energy of equation \eqref{HamPDE} is of the form
$$H(z,z_x)=\sum_\ell |f_\ell(z)|^{p_\ell} |g_\ell(z_x)|^{q_\ell}.$$
However, the method can be defined for problems whose Hamiltonian function involves higher order derivatives \cite{Matsuo}. 
Let be $\widetilde H(\mathbf{Z})$ the vector whose $m$-th entry is a space approximation of $H$ at $x=x_m$ in the form
\begin{equation}\label{locHam}
\widetilde H(\mathbf{Z})_m=\sum_\ell P_\ell(Z_m)Q^+_\ell(Z_m)Q^-_\ell(Z_m)\approx H(z,z_x)\vert_{x=x_m},
\end{equation}
where
\begin{equation}\label{PQ}
P_\ell(Z_m)=|f_\ell(Z_m)|^{p_\ell},\qquad Q_\ell^+(Z_m)=|g_\ell^+(\delta_m^+Z_m)|^{q^+_\ell},\qquad Q_\ell^-(Z_m)=|g_\ell^-(\delta_m^-Z_m)|^{q^-_\ell},
\end{equation}
functions $f_\ell,$ $g_\ell^+$ and $g_\ell^-$ are analytic and  $p_\ell,q_\ell,q_\ell^+,q_\ell^-\in\{2,3,\ldots\}$.
The DVD method approximates equation \eqref{HamPDE} as
\begin{equation}\label{DVDgen}
\delta_n^+z_{m,n}=-\mathrm{i}\mathcal{F}(\mathbf{z}_n,\mathbf{z}_{n+1})_m,
\end{equation}
where
\begin{equation}\label{mathF}
\mathcal{F}(\mathbf{z}_n,\mathbf{z}_{n+1})_m=\left(\frac{\delta \widetilde H}{\delta(\mathbf{z_{n+1}^*},\mathbf{z_{n}^*})}\right)_m
\end{equation}
is a discrete approximation of the variational derivative at time $t=t_n$. Function $\mathcal{F}(\mathbf{a},\mathbf{b})$ is continuous for any value of $(\mathbf{a},\mathbf{b})$ \cite{Dahlby} and is defined as
\begin{equation}\label{discVD}
\left(\frac{\delta \widetilde H}{\delta(\mathbf{a},\mathbf{b})}\right)_m=\left(\frac{\partial \widetilde H}{\partial(\mathbf{a},\mathbf{b})}\right)_m\!\!\!\!-\delta_m^-\left(\frac{\partial \widetilde H}{\partial\delta^+(\mathbf{a},\mathbf{b})}\right)_m\!\!\!\!-\delta_m^+\left(\frac{\partial \widetilde H}{\partial\delta^-(\mathbf{a},\mathbf{b})}\right)_m,
\end{equation}
where the operators at the right hand side are given by
\begin{align}\label{op1}
\left(\frac{\partial \widetilde H}{\partial(\mathbf{a},\mathbf{b})}\right)_{\!m}\!\!\!=&\left(\frac{Q_\ell^+(\mathbf{a}_m)Q_\ell^-(\mathbf{a}_m)+Q_\ell^+(\mathbf{b}_m)Q_\ell^-(\mathbf{b}_m)}2\right)\!\left(\frac{f_\ell(\mathbf{a}_m)-f_\ell(\mathbf{b}_m)}{\mathbf{a}_m-\mathbf{b}_m}\right)\rho_1,\\\label{op2}
\left(\frac{\partial \widetilde H}{\partial\delta^+(\mathbf{a},\mathbf{b})}\right)_{\!m}\!\!\!=&\left(\frac{P_\ell(\mathbf{a}_m)+P_\ell(\mathbf{b}_m)}2\right)\!\left(\frac{Q_\ell^-(\mathbf{a}_m)+Q_\ell^-(\mathbf{b}_m)}2\right)\!\left(\frac{g_\ell^+(\delta_m^+\mathbf{a}_m)-g_\ell^+(\delta_m^+\mathbf{b}_m)}{\delta_m^+\mathbf{a}_m-\delta_m^+\mathbf{b}_m}\right)\rho_2,\\\label{op3}
\left(\frac{\partial \widetilde H}{\partial\delta^-(\mathbf{a},\mathbf{b})}\right)_{\!m}\!\!\!=&\left(\frac{P_\ell(\mathbf{a}_m)+P_\ell(\mathbf{b}_m)}2\right)\!\left(\frac{Q_\ell^+(\mathbf{a}_m)+Q_\ell^+(\mathbf{b}_m)}2\right)\!\left(\frac{g_\ell^-(\delta_m^+\mathbf{a}_m)-g_\ell^-(\delta_m^+\mathbf{b}_m)}{\delta_m^-\mathbf{a}_m-\delta_m^-\mathbf{b}_m}\right)\rho_3,
\end{align}
respectively, with
\begin{align*} 
\rho_1=&\,\rho(p_\ell;f_\ell(\mathbf{a}_m),f_\ell(\mathbf{b}_m)),\\ 
\rho_2=&\,\rho(q_\ell^+;g_\ell^+(\delta^+_m\mathbf{a}_m),g_\ell^+(\delta^+_m\mathbf{b}_m)),\\
\rho_3=&\,\rho(q_\ell^-;g_\ell^-(\delta^-_m\mathbf{a}_m),g_\ell^-(\delta^-_m\mathbf{b}_m)),
\end{align*}
and
$$\rho(k;s_1,s_2)=
\begin{cases} 
\displaystyle{
\frac{s_1^*+s_2^*}2(|s_1|^{k-2}+|s_1|^{k-4}|s_2|^2+\ldots+|s_2|^{k-2})} & \text{if\,\,} k \text{\,\,even,}\\
\displaystyle{\frac{s_1^*+s_2^*}2\frac{|s_1|^{k-1}+|s_1|^{k-2}|s_2|^2+\ldots+|s_2|^{k-1}}{|s_1|+|s_2|}} & \text{if\,\,} k \text{\,\,odd.}
\end{cases} 
$$
Method \eqref{DVDgen}--\eqref{op3} is second order accurate in space and time and when it is complemented by suitable boundary conditions, for example periodic, it conserves the semidiscrete global energy \cite{Matsuo}
\begin{equation}\label{SDHam}
\widetilde{\mathcal{H}}(\mathbf{Z})=\Delta x\sum_m \widetilde H(\mathbf{Z})_m.
\end{equation}
\subsection{Exponentially fitted Discrete Variational Derivative method}
Here we derive an exponentially fitted version of the DVD method \eqref{DVDgen}--\eqref{op3} following an approach that has been similarly used in \cite{DP} for approximating the space derivatives of a diffusion equation. Assuming that the solution of \eqref{HamPDE} is smooth, the continuity of the discrete variational derivative \eqref{discVD} implies that in the limit $\Delta t\rightarrow 0$, method \eqref{DVDgen}--\eqref{op3} converges to the system of ODEs,
\begin{equation}\label{ODEsys}
\mathbf{Z}'(t)=-\mathrm i\mathcal{F}(\mathbf{Z}(t),\mathbf{Z}(t)),
\end{equation}
where the function at the right hand side is well defined due to the smoothness of functions $f_\ell$, $g_\ell^+$ and $g_\ell^-$. 

If the solution of \eqref{ODEsys} oscillates with frequency $\omega$, we look for an approximation of the time derivative at the left hand side of \eqref{ODEsys} requiring that it is exact when the solution belongs to the fitting space $\Theta$ generated by the basis
\begin{equation}\label{basis}
\mathcal{B}_\Theta=\{1,\cos(\omega t),\sin(\omega t)\}.
\end{equation}
 We start from the truncated Taylor expansions
\begin{align*}
u(t+\Delta t)=&\,u(t+\tfrac{\Delta t}2)+\tfrac{\Delta t}2u'(t+\tfrac{\Delta t}2)+\mathcal{O}(\Delta t^2),\\
u(t)=&\,u(t+\tfrac{\Delta t}2)-\tfrac{\Delta t}2u'(t+\tfrac{\Delta t}2)+\mathcal{O}(\Delta t^2),
\end{align*}
yielding
\begin{equation}\label{trunc}
u(t+\Delta t)-u(t)=\Delta t u'(t+\tfrac{\Delta t}2)+\mathcal{O}(\Delta t^2).
\end{equation}
Moreover, equations
\begin{align*}
\Delta t u'(t+\tfrac{\Delta t}2)=&\,{\Delta t} u'(t)+\mathcal{O}(\Delta t^2),\\
\Delta t u'(t+\tfrac{\Delta t}2)=&\,{\Delta t} u'(t+\Delta t)+\mathcal{O}(\Delta t^2),
\end{align*}
and \eqref{trunc} imply that
\begin{equation}\label{standard}
u(t+\Delta t)-u(t)=\tfrac{\Delta t}2 (u'(t+\Delta t)+u'(t))+\mathcal{R}(\Delta t^2).
\end{equation}
The remainder $\mathcal{R}(\Delta t^2)=\mathcal{O}(\Delta t^2)$ when $u$ is a generic function, and it is zero if $u$ belongs to the function space generated by the set $\{1,t,t^2\}$. Our goal is to suitably modify equation \eqref{standard} in order to obtain a formula such that $\mathcal{R}(\Delta t^2)=0$ when evaluated on functions $u\in\Theta$. In particular, we look for a formula of the type
\begin{equation}\label{ef1st}
\alpha u(t+\Delta t)+\beta u(t)=\tfrac{\Delta t}2 (u'(t+\Delta t)+u'(t)),
\end{equation}
that holds true for all $u\in\Theta$ and for two real coefficients $\alpha$ and $\beta$ to be determined. These two parameters are determined by requiring exactness of formula \eqref{ef1st} for all $u\in\mathcal{B}_\Theta$. Substituting $u(t)=1$ in \eqref{ef1st} implies $\alpha=-\beta$. Requiring exactness of \eqref{ef1st} for both $u(t)=\sin(\omega t)$ and $u(t)=\cos(\omega t)$ is equivalent as solving \eqref{ef1st} for $u(t)=\mathrm{e}^{\mathrm{i}\omega t}$, i.e.,
$$\alpha (\mathrm{e}^{\mathrm{i}\omega (t+\Delta t)}-\mathrm{e}^{\mathrm{i}\omega t})=\frac{\mathrm{i} \omega \Delta t}2\{\mathrm{e}^{\mathrm{i}\omega (t+\Delta t)}+\mathrm{e}^{\mathrm{i}\omega t}\},$$
or equivalently,
$$
 \alpha (\mathrm{e}^{\mathrm{i}\omega \Delta t}-1)=\frac{\mathrm{i} \omega \Delta t}2\{\mathrm{e}^{\mathrm{i}\omega \Delta t}+1\},
$$
that yields
\begin{equation}\label{alpha}
\alpha=\frac{\omega \Delta t(1+\cos(\omega \Delta t))}{2\sin(\omega \Delta t)}.
\end{equation}
Therefore, with this approximation of the time derivative, the exponentially fitted version of the DVD method \eqref{DVDgen} proposed in this paper is given by
\begin{equation}\label{EFDVD}
\alpha\delta_n^{+}z_{m,n}=-\mathrm{i}\mathcal{F}(\mathbf{z}_n,\mathbf{z}_{n+1})_m
\end{equation}
with $\alpha$ defined in \eqref{alpha}.
\begin{theorem}
Under suitable boundary conditions, such as periodic, the EF DVD method \eqref{EFDVD} conserves the semidiscrete global energy \eqref{SDHam}.
\end{theorem}
\begin{proof}
The proof follows along similar lines as the one that in \cite{Matsuo} shows that method \eqref{DVDgen} conserves \eqref{SDHam}. In fact,
given the definitions \eqref{locHam}, \eqref{op1}--\eqref{op3}, and repeatedly applying the equality $$\tfrac{1}2(\beta_1-\beta_2)(\beta_3+\beta_4)+\tfrac{1}2(\beta_1+\beta_2)(\beta_3-\beta_4)=\beta_1\beta_3-\beta_2\beta_4,\qquad \beta_1,\beta_2,\beta_3,\beta_4\in\mathbb{C},$$
one has
\begin{align}\label{diffH}
\widetilde{\mathcal H}(\mathbf{z_{n+1}})&\,-\widetilde{\mathcal H}(\mathbf{z_{n}})=\Delta x \sum_m (\widetilde H(\mathbf{z}_{n+1})_m-\widetilde H(\mathbf{z}_{n})_m)\\\nonumber
=&\,\Delta x\sum_m\!\left\{\!\left(\frac{\partial \widetilde{H}}{\partial(\mathbf{z_{n+1}},\mathbf{z_n})}\right)_{\!m}\!\!\!\!({z_{m,n+1}}-z_{m,n})\!+\!\left(\frac{\partial \widetilde{H}}{\partial(\mathbf{z_{n+1}}^*,\mathbf{z_n}^*)}\right)_{\!m}\!\!\!\!({z_{m,n+1}^*}-z_{m,n}^*)\right.\\\nonumber
&+\!\left(\frac{\partial \widetilde{H}}{\partial\delta^+(\mathbf{z_{n+1}},\mathbf{z_n})}\right)_{\!m}\!\!\!\!\delta^+_m({z_{m,n+1}}-z_{m,n})\!+\!\left(\frac{\partial \widetilde{H}}{\partial\delta^+(\mathbf{z_{n+1}}^*,\mathbf{z_n}^*)}\right)_{\!m}\!\!\!\!\delta^+_m({z_{m,n+1}^*}-z_{m,n}^*)\\\nonumber
&\left.+\!\left(\frac{\partial \widetilde{H}}{\partial\delta^-(\mathbf{z_{n+1}},\mathbf{z_n})}\right)_{\!m}\!\!\!\!\delta^-_m({z_{m,n+1}}-z_{m,n})\!+\!\left(\frac{\partial \widetilde{H}}{\partial\delta^-(\mathbf{z_{n+1}}^*,\mathbf{z_n}^*)}\right)_{\!m}\!\!\!\!\delta^-_m({z_{m,n+1}^*}-z_{m,n}^*)\right\}
\end{align}
Summing by parts and assuming that the arising boundary terms vanish, the right hand side of \eqref{diffH} can be equivalently written as
\begin{align}\label{quantity}
\Delta x\sum_m\left\{\left(\frac{\delta \widetilde{H}}{\delta(\mathbf{z_{n+1}},\mathbf{z_n})}\right)_m\!\!\!\! (z_{m,n+1}-z_{m,n})+\left(\frac{\delta \widetilde{H}}{\delta(\mathbf{z_{n+1}^*},\mathbf{z_n}^*)}\right)_m \!\!\!\!(z_{m,n+1}^*-z_{m,n}^*)\right\}
\end{align}
where we have also used definition \eqref{discVD}. Taking into account that $\mathbf{z}_n$ and  $\mathbf{z}_{n+1}$ satisfy \eqref{EFDVD}, definition \eqref{mathF}, 
and observing that 
$$\left(\frac{\delta \widetilde{H}}{\delta(\mathbf{z_{n+1}^*},\mathbf{z_n}^*)}\right)_m\!\!\!\!=\left(\frac{\delta \widetilde{H}}{\delta(\mathbf{z_{n+1}},\mathbf{z_n})}\right)_m^*, $$ 
we can rewrite expression \eqref{quantity} as
\begin{align}
-\mathrm{i}\frac{\Delta t\Delta x}{\alpha}\sum_m\left\{\mathcal{F}(\mathbf{z}_n^*,\mathbf{z}_{n+1}^*)_m\mathcal{F}(\mathbf{z}_n,\mathbf{z}_{n+1})_m-\mathcal{F}(\mathbf{z}_n,\mathbf{z}_{n+1})_m \mathcal{F}(\mathbf{z}^*_n,\mathbf{z}^*_{n+1})_m\right\}=0
\end{align}
Therefore, $$\widetilde{\mathcal{H}}(\mathbf{z_{n+1}})=\widetilde{\mathcal{H}}(\mathbf{z_{n}})$$
follows from \eqref{diffH}.
\end{proof}
\section{Average Vector Field Method}\label{AVFsec}
We introduce here the AVF method for Hamiltonian problems in the form \eqref{genHam}.  
Given a semidiscretization of the Hamiltonian functional \eqref{Ham},
\begin{equation}\label{sdHAVF}
\widetilde{\mathcal{H}}=\Delta x\sum_m \widetilde{H}(\mathbf{Z})_m,
\end{equation}
the AVF method amounts to \cite{Quisp}
\begin{equation}\label{classAVF}
\delta_n^+\mathbf{z}_n=\widetilde{\mathcal{J}}\int_0^1\nabla \widetilde{\mathcal{H}}(\xi\mathbf{z}_{n+1}+(1-\xi)\mathbf{z}_n)\,\mathrm{d}\xi,
\end{equation}
where $\widetilde{\mathcal{J}}$ is a skew-adjoint finite dimensional semidiscretization of $\mathcal{J}$. The AVF method \eqref{classAVF} is second order accurate, and if the boundary conditions are periodic it conserves the semidiscrete global energy \eqref{sdHAVF} \cite{Celle}. More recently it has been proved that, regardless of the specific boundary conditions assigned to the problem, the AVF method preserves the local energy conservation law of the space discretization of \eqref{genHam} \cite{McLachlan,Frasca-Caccia2021}. 

An exponentially fitted version of the AVF method that is exact on functions in the linear space generated by the basis $\mathcal{B}_\Theta$ in \eqref{basis} has been introduced in \cite{Miya} and is defined by
\begin{equation}\label{EFAVF}
\alpha\delta_n^+\mathbf{z}_n=\widetilde{\mathcal{J}}\int_0^1\nabla \widetilde{\mathcal{H}}(\xi\mathbf{z}_{n+1}+(1-\xi)\mathbf{z}_n)\,\mathrm{d}\xi,
\end{equation}
where the parameter $\alpha$ is defined as in \eqref{alpha}. Under suitable assumptions on the boundary conditions, the EF AVF method \eqref{EFAVF} conserves the global energy \eqref{sdHAVF}. The following theorem proves that the energy is conserved locally.
\begin{theorem}\label{theoEFAVF}
The EF AVF method \eqref{EFAVF} has a local energy conservation law.
\end{theorem}
\begin{proof} Following similar steps as those proving the local conservation of the energy of the classical AVF method \eqref{classAVF} in \cite{Frasca-Caccia2021}, we obtain
\begin{align*}
\delta_n^+\widetilde{H}(\mathbf{z}_n)_m=&\,\frac{1}{\Delta t}\int_0^1\!\!\frac{\mathrm{d}}{\mathrm{d}\xi}\widetilde{H}(\xi\mathbf{z}_{n+1}+(1-\xi)\mathbf{z}_n)_m\,\mathrm{d}\xi=\sum_m \delta_n^+z_{m,n}\left.\int_0^1\frac{\partial}{\partial Z_m}\widetilde{H}(\mathbf{Z})\right\vert_{\mathbf{Z}=\xi \mathbf{z}_1+(1-\xi)\mathbf{z}_0}\\
=&\left(\int_0^1\nabla \widetilde{H}(\xi\mathbf{z}_{n+1}+(1-\xi)\mathbf{z}_n)\,\mathrm{d}\xi\right)\delta_n^+\mathbf{z}_n+\delta_m^+\widetilde{F}(\mathbf{z}_n,\mathbf{z}_{n+1})\\
=&\,\frac{1}\alpha\!\int_0^1\!\!\nabla \widetilde{H}(\xi\mathbf{z}_{n+1}+(1-\xi)\mathbf{z}_n)\,\mathrm{d}\xi\,\widetilde{\mathcal{J}}\!\int_0^1\!\!\nabla \widetilde{\mathcal{H}}(\xi\mathbf{z}_{n+1}+(1-\xi)\mathbf{z}_n)\,\mathrm{d}\xi+\delta_m^+\widetilde{F}(\mathbf{z}_n,\mathbf{z}_{n+1})\\
=&\,\delta_m^+\widetilde{F}(\mathbf{z}_n,\mathbf{z}_{n+1}),
\end{align*}
where we omit the expression of the flux. This follows from the application of the summation by parts formula, does not depend on $\alpha$, and can be found in \cite{Frasca-Caccia2021}.
\end{proof}
\begin{remark}
As the conservation law obtained in the proof of Theorem~\ref{theoEFAVF} is independent of $\alpha$, it follows that the solutions of the EF AVF method \eqref{EFAVF} and of the classic AVF method \eqref{classAVF} satisfy the same conservation law.
\end{remark}
\section{Nonlinear Schr\"odinger equation}\label{NLSsec}
The nonlinear Schr\"odinger (NLS) equation for the complex variable $z=u+\mathrm{i}v$,
\begin{equation}\label{compNLS}
\mathrm{i}z_t+z_{xx}+|z|^2z=0,\qquad (x,t)\in(a,b)\times(0,T),
\end{equation}
can be written in Hamiltonian form \eqref{HamPDE}
with Hamiltonian functional \cite{Bodu}
\begin{equation*}
\mathcal{H}=\int H(z,z_x)  \mathrm{d}x = \int \left(|z_x|^2-\tfrac{1}2|z|^4\right) \mathrm{d}x.
\end{equation*} 
Equation \eqref{compNLS} can also be equivalently written as a system of two PDEs for the real variables $u$ and $v$,
\begin{equation}\label{realNLS}
\begin{cases} 
u_t+v_{xx}+(u^2+v^2)v=0,
\\ -v_t+u_{xx}+(u^2+v^2)u={0}.\end{cases}
\end{equation}
System \eqref{realNLS} can be written in Hamiltonian form,
\begin{equation}\label{realHamform}
\left(\begin{array}{c}
u_t\\
v_t
\end{array}\right)=\mathcal{J}\left(\begin{array}{c}
\frac{\delta \mathcal H}{\delta u}\\
\frac{\delta \mathcal H}{\delta v}
\end{array}\right),\qquad \mathcal{J}=\left(\begin{array}{cc}
0 & \tfrac{1}2\\
-\tfrac{1}2 & 0
\end{array}\right),\qquad \mathcal{H}=\int (u_x^2+v_x^2-\tfrac{1}2(u^2+v^2)^2)\,\mathrm{d}x.
\end{equation}
Among the infinitely many conservation laws of the NLS equation we consider here those of the charge and the energy, in the form
\begin{equation}\label{CLAW}
D_x F_\ell+D_t G_\ell=0,\qquad \ell=1,2,
\end{equation}
with
\begin{equation}\label{charge}
F_1=2uv_x-2u_xv,\qquad G_1=u^2+v^2,
\end{equation}
and
\begin{equation}\label{energy} 
F_2=-2u_xu_t-2v_xv_t,\qquad G_2=u^2_x+v^2_x-\frac{1}2(u^2+v^2)^2,
\end{equation}
respectively. 
When suitable boundary conditions, such as periodic, are assigned to system \eqref{realNLS} integration in space of these two conservation laws implies the conservation of the global charge and the global energy
$$\mathcal{M}=\int (u^2+v^2)\,\mathrm{d}x,\qquad \mathcal{H}=\int (u_x^2 + v_x^2 -\tfrac{1}2(u^2+v^2)^2)\,\mathrm{d}x.$$
As shown in \cite{Dahlby}, the two approaches of the AVF method and the DVD method yield the same scheme in many cases. However, when they are applied to the NLS equation two different schemes are obtained. We derive them here separately based on the same definition of the discrete Hamiltonian $\widetilde{H}$ given in \cite{Matsuo}, 
\begin{equation}\label{NLSsdH}
\widetilde{H}(\mathbf Z)_m=\frac{|\delta_m^+Z_m|^2+|\delta_m^-Z_m|^2}2-\frac{1}2|Z_m|^4.
\end{equation}
Setting $Z_m=U_m+\mathrm{i}V_m$, we can rewrite $\widetilde{H}$ equivalently as
\begin{equation}\label{realNLSsdH}
\widetilde{H}(\mathbf U,\mathbf V)_m=\frac{1}2\left\{(\delta_m^+U_{m})^2+(\delta_m^+V_{m})^2+(\delta_m^-U_{m})^2+(\delta_m^-V_{m})^2\right\}-\frac{1}2(U_{m}^2+V_{m}^2)^2.
\end{equation}
\subsection{Discrete Variational Derivative method}
Considering the definition of $\widetilde{H}$ given in \eqref{NLSsdH}, the classical DVD method \eqref{DVDgen} yields the scheme \cite{Matsuo}
\begin{equation}\label{NLSDVD}
\mathrm{i}\delta_n^+z_{m,n}=\delta_m^{(2)}\mu_nz_{m,n}+\mu_n(|z_{n,m}|^2)\mu_nz_{n,m}.
\end{equation}
Setting $z_{m,n}=u_{m,n}+\mathrm{i}v_{m,n}$, method  \eqref{NLSDVD} is equivalent to the following scheme for system \eqref{realNLS}
\begin{equation}\label{classDVD}
\widetilde{A}^{DVD}:=\left(\begin{array}{c}
\delta_n^+u_{m,n}+\delta_m^{(2)}\mu_nv_{m,n}+\mu_n(u_{m,n}^2+v_{m,n}^2)\mu_nv_{m,n}\\
-\delta_n^+v_{m,n}+\delta_m^{(2)}\mu_nu_{m,n}+\mu_n(u_{m,n}^2+v_{m,n}^2)\mu_nu_{m,n}
\end{array}\right)=\mathbf{0}.
\end{equation}
A parametric family of schemes for the NLS equation that have discrete conservation laws of charge and energy has been introduced in \cite{AMC}.  These two discrete conservation laws approximate their continuous counterparts given by \eqref{CLAW} with \eqref{charge} and \eqref{energy}, respectively, and are exactly satisfied by the solutions of the schemes.  As observed in \cite{AMC}, method \eqref{NLSDVD} belongs to this family and its conservation laws are in the form of discrete divergences
\begin{equation}\label{discdiv}
\delta_n^+ \widetilde{G}_\ell + \delta_m^+ \widetilde{F}_\ell, \qquad \ell=1,2,
\end{equation}
that vanish when evaluated on solutions of \eqref{classDVD}. In fact, they can be equivalently written in characteristic form \cite{AMC},
\begin{equation}\label{discCL}
\delta_n^+ \widetilde{G}_\ell + \delta_m^+ \widetilde{F}_\ell=\widetilde{C}_\ell \widetilde{A}^{DVD}, \qquad \ell=1,2,
\end{equation}
with 
\begin{align}\nonumber
\widetilde{F}_1=&\,2(\mu_m\mu_nu_{m-1,n})(\delta_m^-\mu_nv_{m,n})-2(\delta_m^-\mu_nu_{m,n})(\mu_m\mu_nv_{m-1,n}),\\\nonumber
\widetilde{G}_1=&\,u_{m,n}^2+v_{m,n}^2,\qquad \widetilde{C}_1=(2\mu_nu_{m,n},-2\mu_nv_{m,n}),\\\label{FGQ}
\widetilde{F}_2=&\,-2(\delta_m^-\mu_nu_{m,n})(\delta_n^+\mu_mu_{m-1,n})-2(\delta_m^-\mu_nv_{m,n})(\delta_n^+\mu_mv_{m-1,n}),\\\nonumber
\widetilde{G}_2=&\,\tfrac{1}2\left\{(\delta_m^+u_{m,n})^2+(\delta_m^-u_{m,n})^2+(\delta_m^+v_{m,n})^2+(\delta_m^-v_{m,n})^2\right\}-\tfrac{1}2(u_{m,n}^2+v_{m,n}^2)^2,\\\nonumber
\widetilde{C}_2=&\,(-2\delta_n^+v_{m,n},-2\delta_n^+u_{m,n}).
\end{align}
The exponentially fitted version \eqref{EFDVD} of the DVD method \eqref{NLSDVD} is given by
\begin{equation}\label{EFNLSDVD}
\mathrm{i}\alpha\delta_n^+z_{m,n}=\delta_m^{(2)}\mu_nz_{m,n}+\mu_n(|z_{n,m}|^2)\mu_nz_{n,m},
\end{equation}
with $\alpha$ defined according to \eqref{alpha}. With the same notation used in \eqref{classDVD}, method \eqref{EFNLSDVD} is equivalent to
\begin{equation}\label{EFNLSDVDvec}
\widetilde{A}_\alpha^{DVD}:=\left(\begin{array}{c}
\alpha\delta_n^+u_{m,n}+\delta_m^{(2)}\mu_nv_{m,n}+\mu_n(u_{m,n}^2+v_{m,n}^2)\mu_nv_{m,n}\\
-\alpha\delta_n^+v_{m,n}+\delta_m^{(2)}\mu_nu_{m,n}+\mu_n(u_{m,n}^2+v_{m,n}^2)\mu_nu_{m,n}
\end{array}\right)=\mathbf{0}.
\end{equation}
\begin{theorem}\label{theoEFDVD}
The EF DVD method $\widetilde{A}_\alpha^{DVD}$ has discrete conservation laws of charge and energy defined by 
$$\delta_n^+ \widetilde{G}_\ell + \delta_m^+ \widetilde{F}_\ell=\widetilde{C}_\ell \widetilde{A}^{DVD}_\alpha, \qquad \ell=1,2,$$
 with
\begin{equation}\label{EFG1}
\widetilde{G}_1=\alpha(u_{m,n}^2+v_{m,n}^2),
\end{equation}
and functions $\widetilde{F}_1,\widetilde{C}_1,\widetilde{F}_2,\widetilde{G}_2,$ and $\widetilde{C}_2,$ defined as in \eqref{FGQ}.
\end{theorem}
\begin{proof}
As the only difference between schemes $\widetilde{A}_\alpha^{DVD}$ and $\widetilde{A}^{DVD}$ is the factor $\alpha$ multiplying the forward difference approximations of the time derivative, we only need to investigate how the introduction of this factor effects the conservation laws of $\widetilde{A}^{DVD}$. 

Product $\widetilde{C}_2\widetilde{A}_\alpha^{DVD}$ is not affected by the value of $\alpha$. In fact, expanding it one obtains that
$$-2\alpha(\delta_n^+{v_{m,n}})(\delta_n^+{u_{m,n}})+2\alpha(\delta_n^+{u_{m,n}})(\delta_n^+{v_{m,n}})=0,$$
and there is no other term that depends on $\alpha$. So method $\widetilde{A}_\alpha^{DVD}$ has the same energy conservation law of the classic DVD method $\widetilde{A}^{DVD}$ obtained in \cite{AMC}.

Expanding the product $\widetilde{C}_1\widetilde{A}_\alpha^{DVD}$ one obtains that the parameter $\alpha$ only appears in
$$2\alpha(\mu_nu_{m,n})(\delta_n^+u_{m,n})-2\alpha(\mu_nv_{m,n})(\delta_n^+v_{m,n})=\delta_n^+(\alpha(u_{m,n}^2+v_{m,n}^2)),$$
defining the density of the charge conservation law of the exponentially fitted method as in \eqref{EFG1}. As $\alpha$ does not multiply any other term, the expression of the flux is the same as that of the classic DVD method and it is given by $\widetilde F_1$ in \eqref{FGQ}.
\end{proof}
\begin{remark} 
Under suitable boundary conditions, e.g. periodic, summation in space of the obtained local conservation laws implies that the DVD method \eqref{classDVD} and the EF DVD method \eqref{EFNLSDVDvec} conserve the global charge,
\begin{equation}\label{totch}
\widetilde{\mathcal{M}}_n=\,\Delta x\sum_m (u_{m,n}^2+v_{m,n}^2),\\
\end{equation}
and the global energy,
\begin{equation}\label{totH}
\widetilde{\mathcal{H}}_n=\,\Delta x\sum_m \left\{\tfrac{1}2\left[(\delta_m^+u_{m,n})^2+(\delta_m^-u_{m,n})^2+(\delta_m^+v_{m,n})^2+(\delta_m^-v_{m,n})^2\right]-\tfrac{1}2(u_{m,n}^2+v_{m,n}^2)^2\right\}.
\end{equation}
\end{remark}
\subsection{Average Vector Field method}
With the approximation \eqref{sdHAVF} and \eqref{realNLSsdH} of the Hamiltonian functional, the AVF method \eqref{classAVF} approximates system \eqref{realHamform} as
\begin{equation}\label{NLSAVFvec}
\widetilde{A}^{AVF}:=\left(\begin{array}{c}
\delta_n^+u_{m,n}+\delta_m^{(2)}\mu_nv_{m,n}+\mu_n(v_{m,n}^2+\frac{2}3u_{m,n}^2)\mu_nv_{m,n}+\frac{1}3\mu_n(u_{m,n}^2v_{m,n})\\
-\delta_n^+v_{m,n}+\delta_m^{(2)}\mu_nu_{m,n}+\mu_n(u_{m,n}^2+\frac{2}3v_{m,n}^2)\mu_nu_{m,n}+\frac{1}3\mu_n(v_{m,n}^2u_{m,n})
\end{array}\right)=\mathbf{0}.
\end{equation}
Similarly, the approximation given by the exponentially fitted AVF method \eqref{EFAVF} amounts to
\begin{equation}\label{EFNLSAVFvec}
\widetilde{A}^{AVF}_\alpha:=\left(\begin{array}{c}
\alpha\delta_n^+u_{m,n}+\delta_m^{(2)}\mu_nv_{m,n}+\mu_n(v_{m,n}^2+\frac{2}3u_{m,n}^2)\mu_nv_{m,n}+\frac{1}3\mu_n(u_{m,n}^2v_{m,n})\\
-\alpha\delta_n^+v_{m,n}+\delta_m^{(2)}\mu_nu_{m,n}+\mu_n(u_{m,n}^2+\frac{2}3v_{m,n}^2)\mu_nu_{m,n}+\frac{1}3\mu_n(v_{m,n}^2u_{m,n})
\end{array}\right)=\mathbf{0},
\end{equation}
with $\alpha$ given in \eqref{alpha}.

\begin{theorem}
The classic AVF method \eqref{NLSAVFvec} and the EF AVF method \eqref{EFNLSAVFvec} satisfy the same energy conservation law of the classic DVD method  \eqref{NLSDVD} and of the EF DVD method \eqref{EFNLSDVD}, defined by 
$$\delta_m^+ \widetilde{F}_2+\delta_n^+\widetilde{G}_2=\widetilde{{C}}_2\widetilde{A}^{AVF}=\widetilde{{C}}_2\widetilde{A}^{AVF}_\alpha=\widetilde{{C}}_2\widetilde{A}^{DVD}=\widetilde{{C}}_2\widetilde{A}^{DVD}_\alpha,$$
with functions $\widetilde{F}_2$, $\widetilde{G}_2$ and $\widetilde{{C}}_2$ given in \eqref{FGQ}.
\end{theorem}
\begin{proof}
The statement can be proved by expanding the calculations outlined in Theorem \ref{theoEFAVF}. However, we instead evaluate the differences 
$$\widetilde{A}^{AVF}-\widetilde{A}^{DVD}=\widetilde{A}^{AVF}_\alpha-\widetilde{A}^{DVD}_\alpha=\frac{\Delta t}6\left(\begin{array}{c}
\delta_n^+u_{m,n}(u_{m,n+1}v_{m,n}-u_{m,n}v_{m,n+1})\\
-\delta_n^+v_{m,n}(u_{m,n+1}v_{m,n}-u_{m,n}v_{m,n+1})
\end{array}\right).
$$
Since, with $\widetilde C_2$ defined in \eqref{FGQ},
$$\widetilde{C}_2\left(\begin{array}{c}
\delta_n^+u_{m,n}(u_{m,n+1}v_{m,n}-u_{m,n}v_{m,n+1})\\
-\delta_n^+v_{m,n}(u_{m,n+1}v_{m,n}-u_{m,n}v_{m,n+1})
\end{array}\right)=\mathbf{0},$$
it follows from \eqref{discCL} and Theorem \ref{theoEFDVD} that 
$$\delta_m^+ \widetilde{F}_2+\delta_n^+ \widetilde{G}_2=\widetilde{C}_2 \widetilde{A}^{DVD}=\widetilde{C}_2 \widetilde{A}^{DVD}_\alpha=\widetilde{C}_2 \widetilde{A}^{AVF}=\widetilde{C}_2 \widetilde{A}^{AVF}_\alpha,$$
with $\widetilde{G}_2$ and $\widetilde{F}_2$ given in \eqref{FGQ}.
\end{proof}
\section{Numerical tests}\label{testsec}
As a benchmark problem to compare the numerical methods described in this paper and to test their conservative properties, we consider here the breather solution \cite{AEK},
\begin{figure}[tbp]
\begin{center}
\includegraphics[width=.32\textwidth,height=5cm]{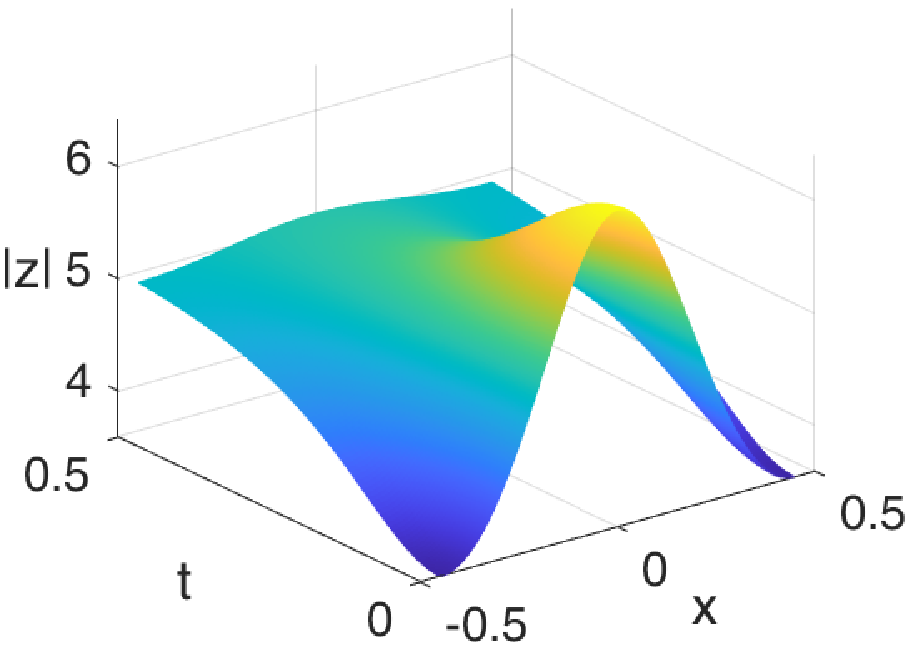}
\includegraphics[width=.32\textwidth,height=5cm]{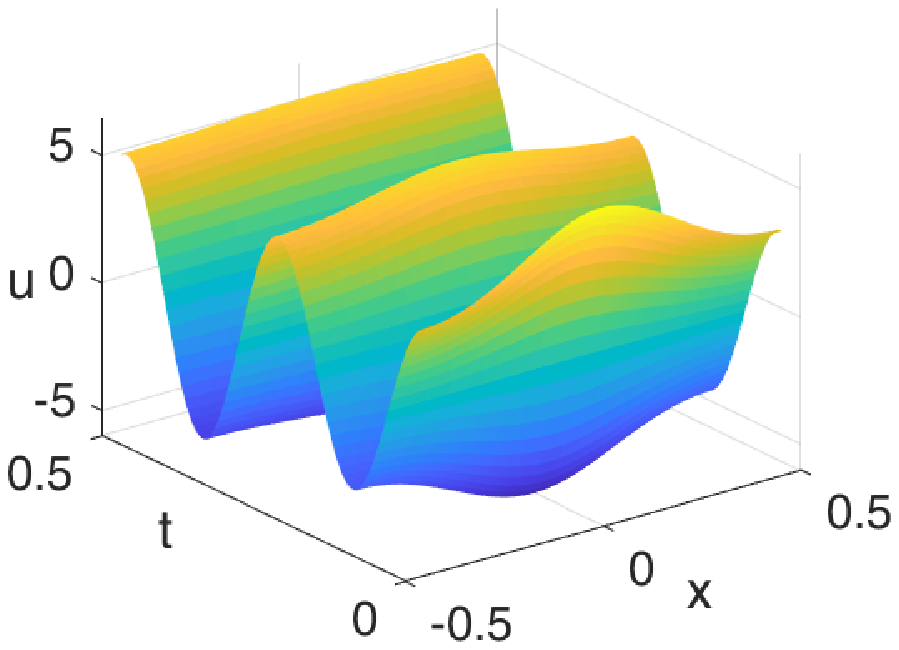}
\includegraphics[width=.32\textwidth,height=5cm]{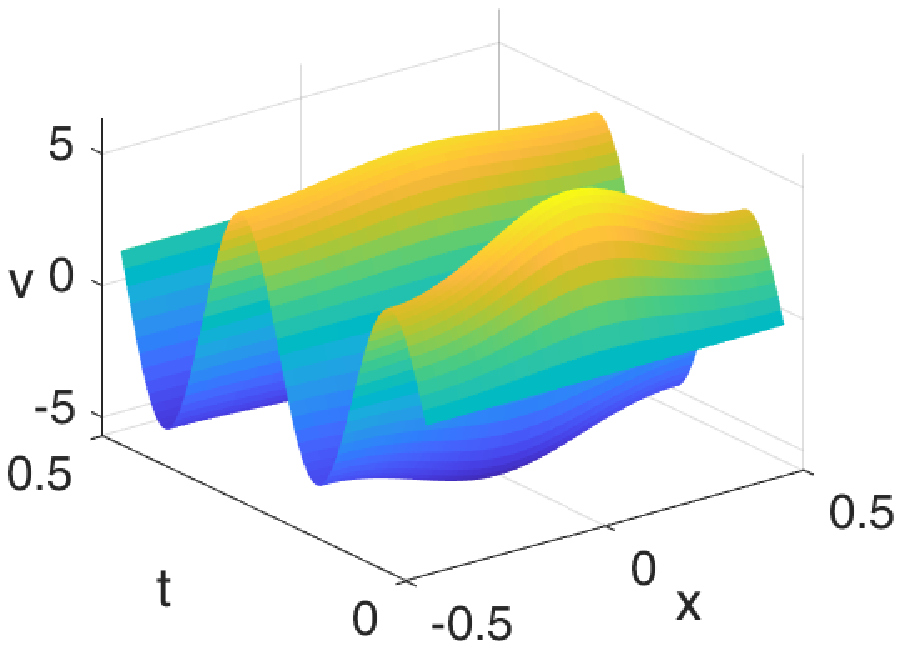}
\end{center}
	\caption{Solution of the breather problem for NLS: $|z(x,t)|$ (left), $u(x,t)$ (centre), $v(x,t)$ (right)}
	\label{fig:NLSbr}
\end{figure} 
\begin{align}\label{exbr}
z(x,t)&=\left(\frac{2\beta^2\cosh\theta+2\mathrm{i}\beta\sqrt{2-\beta^2}\sinh\theta}{2\cosh\theta-\sqrt{4-2\beta^2}\cos(\sqrt\omega\beta x)}-1\right)\sqrt\omega\mathrm{e}^{\mathrm{i}\omega t},\quad \theta=\omega\beta\sqrt{2-\beta^2}t,\quad\beta<\sqrt{2},\\\nonumber
u(x,t)&=\operatorname{Re}(z),\qquad v(x,t)=\operatorname{Im}(z). 
\end{align}
As in \cite{CFC}, we consider the restriction of this solution to the domain $(x,t)\in [-\pi/7,\pi/7]\times[0,0.5]$ and we set $\beta=1.4$. Figure~\ref{fig:NLSbr} shows a graph of the exact solution. The initial condition is obtained from formula \eqref{exbr} evaluated at $t=0$. The frequency of oscillation of $u$ and $v$ is given by $\omega$ and it can be derived from the initial condition.  We set here $\omega=25$. The numerical methods are solved on uniform grids defined by $\Delta x = 2\pi/7000,$ and $\Delta t_k = 0.01/2^k,\, k = 0, \ldots , 5$.

As the computational cost of all methods is similar, we compare them on the basis of the error in their solution at the final time $t=t_N$, evaluated as
$${\rm Sol\,\, err}=\sqrt{\frac{\|\mathbf{u}_N-u(\mathbf{x},t_N)\|^2+\|\mathbf{v}_N-v(\mathbf{x},t_N)\|}{\|u(\mathbf{x},t_N)\|^2+\|v(\mathbf{x},t_N)\|^2}}.$$
We investigate the convergence of the schemes by estimating the order of accuracy of the time integrator as
$${\rm Order}=\log_2\left(\frac{{\rm Sol\,\, err}_{{k-1}}}{{\rm Sol\,\, err}_k}\right),\qquad k=1,\ldots, 5,$$
where Sol err$_k$ denotes the error in the solution obtained with time step $\Delta t_k$. 

For fixed $\Delta x$ this estimate of the order of convergence is valid only for $k$ small enough, so that $\Delta t_k$ is large enough compared to $\Delta x$, and the leading term of error is proportional to $\Delta t_k$. Hence, in the following tables the symbol ``***'' means that the method has converged to a solution whose time component of the error is negligible compared to the spatial one.

The error in the local conservation laws is evaluated as
$$
{\rm Err}_\ell=\max_{m,n}\left(\{D_{\Delta x}\widetilde{F}_\ell+D_{\Delta t}\widetilde{G}_\ell\}\big\vert_{(x_m,t_n)}\right),\qquad \ell=1,2,
$$
where for all schemes functions $\widetilde{F}_1$, $\widetilde{G}_2$ and $\widetilde{F}_2$ are those given in \eqref{FGQ} and function $\widetilde{G_1}$ is defined in \eqref{EFG1} for the EF DVD method \eqref{EFNLSDVDvec}, or in \eqref{FGQ} for all other methods. 

Since the solution of this problem satisfies periodic boundary conditions, the global charge and Hamiltonian are conserved. We evaluate the error in these two invariants as
\begin{equation*}
{\rm Err}_{\mathcal{M}}=\max_{n}\left|\widetilde{\mathcal M}_n-\widetilde{\mathcal M}_0\right|,\qquad {\rm Err}_{\mathcal{H}}=\max_{n}\left|\widetilde{\mathcal H}_n-\widetilde{\mathcal H}_0\right|,
\end{equation*}
with $\widetilde{\mathcal{M}}_n$ and $\widetilde{\mathcal{H}}_n$ defined as in \eqref{totch} and \eqref{totH}, respectively.  
\begin{table}[t]
\caption{Error in local conservation laws}\label{tab:NLSCL}
\small
\begingroup
\setlength{\tabcolsep}{6pt} 
\renewcommand{\arraystretch}{1.12} 
\centerline{\begin{tabular}{|c||c|c||c|c||c|c||c|c|}
\hline
&  \multicolumn{2}{|c||}{Classic DVD \cite{Matsuo}} & \multicolumn{2}{|c||}{EF DVD} &\multicolumn{2}{|c||}{Classic AVF \cite{Quisp}} & \multicolumn{2}{|c|}{EF AVF \cite{Miya}}\\ 
\hline
$n$ & Err$_1$ & Err$_2$ & Err$_1$ & Err$_2$ & Err$_1$ & Err$_2$ & Err$_1$ & Err$_2$\\
\hline
0& 2.68e-08 & 8.05e-07 & 2.92e-08 & 7.51e-07 & 1.33e+00 & 6.89e-07 & 8.83e-01 & 7.91e-07\\
1& 3.35e-08 & 1.13e-06 & 3.69e-08 & 9.62e-07 & 2.67e-01 & 8.22e-07 & 1.82e-01 & 9.67e-07\\
2& 3.55e-08 & 8.94e-07 & 3.65e-08 & 9.20e-07 & 6.21e-02 & 1.00e-06 & 4.24e-02 & 9.79e-07\\
3& 4.09e-08 & 1.10e-06 & 4.41e-08 & 9.58e-07 & 1.52e-02 & 1.05e-06 & 1.04e-02 & 1.12e-06\\
4& 3.68e-08 & 1.14e-06 & 3.95e-08 & 1.31e-06 & 3.79e-03 & 1.14e-06 & 2.59e-03 & 1.13e-06\\
5& 4.27e-08 & 1.24e-06 & 5.01e-08 & 1.21e-06 & 9.48e-04 & 1.14e-06 & 6.47e-04 & 1.09e-06 \\
\hline
\end{tabular}}
\endgroup
\end{table}
\begin{table}[t]
\caption{Error in global invariants}\label{tab:NLSinv}
\small
\begingroup
\setlength{\tabcolsep}{6pt} 
\renewcommand{\arraystretch}{1.12}
\centerline{\begin{tabular}{|c||c|c||c|c||c|c||c|c|}
\hline
&  \multicolumn{2}{|c||}{Classic DVD \cite{Matsuo}} & \multicolumn{2}{|c||}{EF DVD} &\multicolumn{2}{|c||}{Classic AVF \cite{Quisp}} & \multicolumn{2}{|c|}{EF AVF \cite{Miya}}\\ 
\hline
$n$ & Err$_{\mathcal{M}}$ & Err$_{\mathcal{H}}$ & Err$_{\mathcal{M}}$ & Err$_{\mathcal{H}}$ & Err$_{\mathcal{M}}$ & Err$_{\mathcal{H}}$ & Err$_{\mathcal{M}}$ & Err$_{\mathcal{H}}$\\
\hline
0& 1.14e-13 & 2.27e-12 & 1.14e-13 & 2.56e-12 & 1.84e-02 & 2.44e-12 & 1.85e-02 & 2.05e-12\\
1& 3.69e-12 & 1.42e-13 & 1.88e-12 & 1.21e-13 & 4.75e-03 & 3.41e-12 & 4.77e-03 & 4.15e-12\\
2& 1.25e-12 & 9.95e-14 & 1.76e-12 & 1.07e-13 & 1.20e-03 & 2.67e-12 & 1.20e-03 & 2.61e-12\\
3& 8.88e-14 & 1.19e-12 & 7.11e-14 & 1.25e-12 & 3.00e-04 & 9.09e-13 & 3.00e-04 & 1.59e-12\\
4& 7.82e-14 & 1.31e-12 & 7.82e-14 & 1.14e-12 & 7.50e-05 & 1.48e-12 & 7.50e-05 & 1.36e-12 \\
5& 8.17e-14 & 1.31e-12 & 8.17e-14 & 1.08e-12 & 1.87e-05 & 1.71e-12 & 1.87e-05 & 1.14e-12 \\
\hline
\end{tabular}}
\endgroup
\end{table}

\begin{table}[t]
\caption{Order of convergence and error in solution}\label{tab:NLSerr}
\small
\centering
\begingroup
\setlength{\tabcolsep}{6pt} 
\renewcommand{\arraystretch}{1.12} 
\begin{tabular}{|c||c|c||c|c||c|c||c|c|}
\hline
&  \multicolumn{2}{|c||}{Classic DVD \cite{Matsuo}} & \multicolumn{2}{|c||}{EF DVD} &\multicolumn{2}{|c||}{Classic AVF \cite{Quisp}} & \multicolumn{2}{|c|}{EF AVF \cite{Miya}}\\ 
\hline
$n$ & Sol err & Order & Sol err & Order & Sol err & Order & Sol err & Order\\
\hline
0& 7.32e-02 &  & 8.79e-03 & & 2.01e-01 &  & 2.06e-01 & \\
1& 1.84e-02 & {1.99} & 2.15e-03 & {2.03} & 1.29e-01 & {0.64} & 1.22e-01 & {0.76}\\
2& 4.55e-03 & 2.01 & 4.89e-04 & 2.13 & 3.80e-02 & 1.76 & 3.56e-02 & 1.78\\
3& 1.09e-03 & 2.06 & 1.50e-04 & 1.70 & 9.71e-03 & 1.97 & 9.06e-03 & 1.97 \\
4& 2.58e-04 & 2.07 & 1.52e-04 & *** & 2.33e-03 & 2.06 & 2.16e-03 & 2.07\\
5 & 1.49e-04 & *** & 1.63e-04 & *** & 4.65e-04 & 2.32 &  4.19e-04 & 2.37 \\
\hline
\end{tabular}
\endgroup
\end{table}
In Table~\ref{tab:NLSCL} and Table~\ref{tab:NLSinv} we show the errors in the local conservation laws and in the global invariants, respectively. All methods preserve the energy conservation law and the global energy. The errors in the table are affected by accumulation of the round-offs and by the approximate solution of the nonlinear schemes by Newton's method and are roughly equal for all values of $\Delta t$. 

The solutions of classic AVF and EF AVF do not satisfy the conservation law of the charge, and do not conserve the global charge. The corresponding errors decrease with the time step and approach zero with the same rate of convergence of the schemes.

In Table~\ref{tab:NLSerr} we show the error in the solution and the estimated rate of convergence of the four considered schemes. All scheme converge with accuracy of the second order, until the space component of the error ($\sim 1.50\text{e-04}$) prevails. 

Figure~\ref{fig:NLSord} shows a logarithmic plot of the solution error against $\Delta t$ illustrating the rate of convergence of all the methods.

Compared to the AVF methods, the DVD methods not only conserve the charge locally and globally but also are more accurate for all values of $\Delta t$. 

The EF AVF method proposed in \cite{Miya} is not substantially more accurate than the classic AVF method. Instead, the EF DVD method introduced here is about one and two orders of magnitude more accurate than the classic DVD and the AVF methods, respectively. The new method is also the one that achieves the maximum possible accuracy in the solution (attainable with the chosen value of $\Delta x$) with the largest time step ($n=3$).

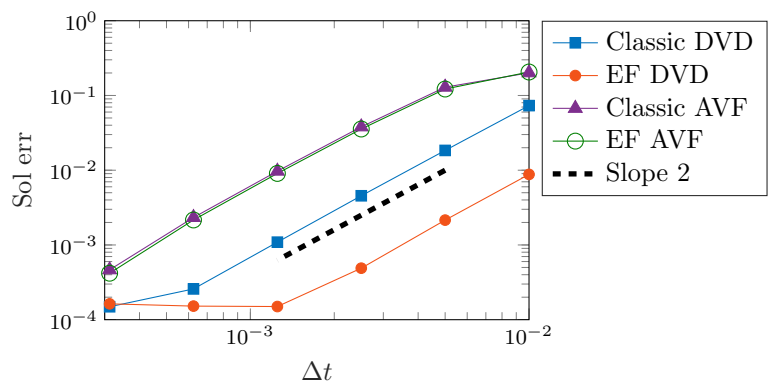
\begin{figure}[tbp]
\begin{center}
	\begin{tikzpicture}

\begin{axis}[
width=2.221in,
height=1.566in,
at={(0.in,0.481in)},
scale only axis,
xmode=log,
xmin=0.0003,
xmax=0.01,
xminorticks=true,
ymode=log,
ymin=0.0001,
ymax=1,
yminorticks=true,
xlabel style={font=\color{white!15!black}},
xlabel=$\Delta t$,
ylabel style={font=\color{white!15!black}},
ylabel=Sol err,
axis background/.style={fill=white},
legend style={legend cell align=left, align=left, draw=white!15!black,legend pos=outer north east}
]
\addplot [color=colorclassyblue, mark=square*, mark options={solid, colorclassyblue},mark size=2pt]
  table[row sep=crcr]{
0.01	0.0732\\
0.005	0.0184\\
0.0025	0.00455\\
0.00125	0.00109\\
0.000625	0.000258\\
0.0003125	0.000149\\
};
\addlegendentry{Classic DVD}

\addplot [color=colorclassyorange, mark=*, mark options={solid, colorclassyorange},mark size=2pt]
  table[row sep=crcr]{
0.01	0.00879\\
0.005	0.00215\\
0.0025	0.000489\\
0.00125	0.000150\\
0.000625	0.000152\\
0.0003125	0.000163\\
};
\addlegendentry{EF DVD}

\addplot [color=colorpurple, mark=triangle*, mark options={solid, colorpurple},mark size=3pt]
  table[row sep=crcr]{
0.01	0.201\\
0.005	0.129\\
0.0025	0.0380\\
0.00125	0.00971\\
0.000625	0.00233\\
0.0003125	0.000465\\
};
\addlegendentry{Classic AVF}

\addplot [color=colorimag, mark=o, mark options={solid, colorimag},mark size=3pt]
  table[row sep=crcr]{
0.01	0.206\\
0.005	0.122\\
0.0025	0.0356\\
0.00125	0.00906\\
0.000625	0.00216\\
0.0003125	0.000419\\
};
\addlegendentry{EF AVF}

\addplot [color=black, dashed, line width=2]
  table[row sep=crcr]{
0.005	0.01\\
0.0025	0.0025\\
0.00125	0.000625\\
};
\addlegendentry{Slope 2}
\end{axis}
\end{tikzpicture}
\end{center}
	\caption{Solution error for NLS breather (logarithmic scale on both axis)}
	\label{fig:NLSord}
\end{figure}

We conclude this section remarking that this problem has been solved in \cite{CFC} by a new second order exponentially fitted method that preserves the local conservation laws of charge and momentum of system \eqref{realNLS}. Comparing with the solution errors reported in \cite{CFC}, we observe that the EF DVD method introduced here is the most accurate and gives errors in the solution that are about ten times smaller.

\section{Conclusion}\label{concl}
In this paper we have introduced a new exponentially fitted version of the discrete variational derivative method in \cite{Matsuo} for complex Hamiltonian PDEs. We have proved that when applied to the nonlinear Schr\"odinger equation this method has local conservation laws of charge and energy that approximate the continuous ones. 

In a more general setting, for real or complex Hamiltonian PDEs, we have proved that the exponentially fitted AVF method introduced by Miyatake \cite{Miya} has the same local conservation law of the energy of the classic AVF method. However, neither the AVF method nor its exponentially fitted version conserve the charge. 

The four considered methods have been applied to a problem whose solution is a breather wave that oscillates with known frequency. The conservative properties of all schemes have been tested and the proposed EF DVD method is the one that performs better.

\subsection*{Acknowledgements} 
This work is supported by GNCS-INDAM project and by PRIN2017-MIUR project. The authors are members of the INdAM research group GNCS.
\bibliographystyle{acm}
\bibliography{bibfile}
\end{document}